\documentclass[british]{amsart}

\usepackage[T1]{fontenc}
\usepackage{babel}
\usepackage{prettyref}
\usepackage[colorlinks=true]{hyperref}

\numberwithin{equation}{section}

\theoremstyle{plain}
\newtheorem{thm}{Theorem}[section]
\newtheorem{conjecture}[thm]{Conjecture}
\newtheorem{prop}[thm]{Proposition}
\newtheorem{cor}[thm]{Corollary}

\theoremstyle{definition}
\newtheorem{defn}[thm]{Definition}

\theoremstyle{remark}
\newtheorem{rem}[thm]{Remark}
\newtheorem{notation}[thm]{Notation}

\newrefformat{prop}{{Proposition~\ref{#1}}}
\newrefformat{lem}{{Lemma~\ref{#1}}}
\newrefformat{cor}{{Corollary~\ref{#1}}}
\newrefformat{conj}{{Conjecture~\ref{#1}}}
\newrefformat{eq}{{Equation~\eqref{#1}}}
\newrefformat{sub}{{\S\eqref{#1}}}

\newcommand\Cb{\mathbb{C}}
\newcommand\N{\mathbb{N}}
\newcommand\Q{\mathbb{Q}}
\newcommand\oQ{\overline{\mathbb{Q}}}
\newcommand\R{\mathbb{R}}
\newcommand\Z{\mathbb{Z}}
\newcommand\Oc{\mathcal{O}}
\newcommand\wht{h}
\DeclareMathOperator\ld{lin.d.}
\DeclareMathOperator\td{tr.deg.}

\title{Polynomial exponential equations and Zilber's conjecture}
\author{V.\ Mantova}
\address{Scuola Normale Superiore, Piazza dei Cavalieri 7, 56126 Pisa, Italy}
\email{vincenzo.mantova@sns.it}
\date{24th November 2015}

\begin{document}

\maketitle

\begin{center}
  With an Appendix by V.\ Mantova and U.\ Zannier
\end{center}

\begin{abstract}
  Assuming Schanuel's conjecture, we prove that any polynomial exponential
  equation in one variable must have a solution that is transcendental over a
  given finitely generated field. With the help of some recent results in
  Diophantine geometry, we obtain the result by proving (unconditionally) that
  certain polynomial exponential equations have only finitely many rational
  solutions.

  This answers affirmatively a question of David Marker, who asked, and proved
  in the case of algebraic coefficients, whether at least the one-variable case
  of Zilber's strong exponential-algebraic closedness conjecture can be reduced
  to Schanuel's conjecture.
\end{abstract}

\section{The problem}

Based on model-theoretic arguments, Zilber conjectured in \cite{Zilber2005} that
the complex exponential function satisfies two strong properties about its
algebraic behaviour. One is the long standing Schanuel's Conjecture, today
considered out of reach, while the other, called ``Strong Exponential-algebraic
Closedness'' or ``Strong Exponential Closedness'', states that all systems of
polynomial-exponential equations compatible with Schanuel's conjecture have
solutions of maximal transcendence degree. Zilber proved that this would imply
that the complex exponential function has a good algebraic description in a very
strong model-theoretic sense (most importantly, the structure would be
axiomatizable with an uncountably categorical sentence).

While Schanuel's Conjecture is currently considered out of reach, except for
some very special known instances, the second property is still relatively
unexplored. Marker observed in \cite{Marker2006} that, at least in some cases,
the second property would already follow from Schanuel's Conjecture. This
suggests investigating whether the second property is actually a consequence of
Schanuel's Conjecture, thereby implying that Schanuel's and Zilber's conjectures
are equivalent.

In the case of systems in one variable the question of Zilber has a particularly
simple shape. A system of equations in one variable compatible with Schanuel's
Conjecture is just an equation
\[
  p(z,\exp(z))=0
\]
where $p(x,y)\in\Cb[x,y]$ is an irreducible polynomial where both $x$ and $y$
appear, namely such that
$\frac{\partial p}{\partial x},\frac{\partial p}{\partial y}\neq0$.

Zilber's Strong Exponential Closedness for one variable asserts the following:

\begin{conjecture}[\cite{Zilber2005}]
  \label{conj:main}For any finitely generated field $k\subset\Cb$ and for any
  irreducible $p(x,y)\in k[x,y]$ such that
  $\frac{\partial p}{\partial x},\frac{\partial p}{\partial y}\neq0$, there
  exists $z\in\Cb$ such that $p(z,\exp(z))=0$ and $\td_{k}(z,\exp(z))=1$.
\end{conjecture}

It is well known and not too difficult to prove that the equation
$p(z,\exp(z))=0$ has infinitely many solutions (see \cite{Marker2006}).  Marker
proved that Schanuel's Conjecture implies \prettyref{conj:main} when
$p\in\oQ[x,y]$, and he asked whether the same holds for any $p\in\Cb[x,y]$
\cite{Marker2006}. Günaydin suggested a different technique in
\cite{Gunaydin2011} and provided some steps towards this
generalisation. Starting from \cite{Gunaydin2011}, we apply a Diophantine result
about function fields and exponential equations from \cite{Zannier2004} and
ultimately give a positive answer.

\begin{thm}
  \label{thm:main}If Schanuel's Conjecture holds, then \prettyref{conj:main} is
  true.
\end{thm}

In fact, using Schanuel's conjecture it is easy to reduce \prettyref{conj:main}
to the problem of counting rational solutions of certain polynomial-exponential
equations. Here we obtain the finiteness of such solutions.
\begin{thm}
  \label{thm:finite}Let $p(x,y)\in\Cb[x,y]$ be an irreducible polynomial such
  that $\frac{\partial p}{\partial x}\neq0$,
  $\frac{\partial p}{\partial y}\neq0$, and let ${\bf b}\in\Cb^{l}$ be a vector
  of $\Q$-linearly independent complex numbers. Then the equation
  \[
    p({\bf x}\cdot{\bf b},\exp({\bf x}\cdot{\bf b}))=0
  \]
  has only finitely many solutions ${\bf x}\in\Q^{l}$.
\end{thm}

For simplicity, we shall first prove \prettyref{thm:finite} assuming Schanuel's
conjecture, as this will suffice to prove \prettyref{thm:main}.  We note that
the proof uses only the fact that $\exp$ is a homomorphism from $(\mathbb{C},+)$
to $(\mathbb{C}^{*},\cdot)$ with cyclic kernel, and the Schanuel property in a
few places. This implies that the conclusion of \prettyref{thm:main} holds in
any exponential field having standard kernel, satisfying Schanuel's Conjecture
and where $p(z,\exp(z))=0$ has infinitely many solutions. Similarly, the
conclusion of \prettyref{thm:finite} holds in any exponential field having
standard kernel and satisfying the conclusion of Baker's theorem on
logarithms. In particular, in the axiomatization of Zilber fields it is
sufficient to state that $p(z,\exp(z))=0$ has infinitely many solutions, rather
than requiring the solutions to have maximal transcendence degree.

In an appendix written with U.\ Zannier, we shall provide a different
unconditional proof of \prettyref{thm:finite} based on Baker's theorem on
logarithms in place of Schanuel's conjecture.

\subsection*{Acknowledgements}

The author would like to thank Umberto Zannier for the clever suggestions that
made this work possible and for his substantial contributions to this paper. The
author would also like to thank Ayhan Günaydin, Amador Martin-Pizarro, Paola
D'Aquino, Antongiulio Fornasiero and Giuseppina Terzo for the several
discussions on this subject which led to this paper and to several improvements
as well. Finally, the author would like to thank an anonymous referee for the
useful comments and corrections.

This work was supported by the FIRB2010 ``New advances in the Model Theory of
exponentiation'' RBFR10V792 and by the ERC AdG ``Diophantine Problems'' 267273.

\section{Proof}

Let $p(x,y)\in\Cb[x,y]$ be an irreducible polynomial over $\Cb$ such that
$\frac{\partial p}{\partial x},\frac{\partial p}{\partial y}\neq0$ and let
$k\subset\Cb$ be a finitely generated field. We shall prove that Schanuel's
Conjecture implies that the equation
\begin{equation}
  p(z,\exp(z))=0\label{eq:orig}
\end{equation}
has only finitely many solutions $z$ such that $z\in\overline{k}$, where
$\overline{k}$ denotes the algebraic closure of $k$ in $\Cb$.  Since the
equation has infinitely many solutions in $\Cb$, this easily implies the desired
statement.

This is done in four steps:

\begin{enumerate}
\item as explained in \cite{Marker2006}, if Schanuel's conjecture is true, then
  there is a finite-dimensional $\Q$-vector space $L\subset\Cb$ containing all
  the $z\in\overline{k}$ such that $p(z,\exp(z))=0$; if ${\bf b}$ is a basis of
  $L$ as a $\Q$-vector space, all such solutions can be written as a scalar
  product ${\bf q}\cdot{\bf b}$ with ${\bf q}\in\Q^{n}$
  (\prettyref{sub:red-to-linear});
\item a special case of the main result of \cite{Gunaydin2011} says that there
  is an $N\in\N$ such that actually ${\bf q}\in\Z\left[\frac{1}{N}\right]$ for
  all the solutions; hence, our problem reduces to the one of counting the
  \emph{integer} solutions ${\bf n}\in\Z^{n}$ of an equation of the form
  $p({\bf x}\cdot{\bf b}/N,\exp({\bf x}\cdot{\bf b}/N))=0$
  (\prettyref{sub:red-to-integer});
\item as suggested by Zannier, a function field version of a theorem of M.\
  Laurent \cite{Zannier2004} lets us reduce to the case where
  $\exp({\bf n}\cdot{\bf b})$ is always algebraic for any solution ${\bf n}$
  (\prettyref{sub:red-to-alg});
\item finally, if $2\pi i$ is in $L$, we specialise it to $0$; assuming
  Schanuel's Conjecture, some arithmetic on $\oQ$ is sufficient to prove
  finiteness (\prettyref{sub:finiteness}).
\end{enumerate}

Step (4) can be replaced with a more complicated, but unconditional, argument
where Baker's theorem on logarithms is used in place of Schanuel's Conjecture;
its details are given in the appendix. Steps (2) and (3) are obtained
unconditionally.

\subsection{\label{sub:red-to-linear}Reduction to linear spaces}

Let us recall Schanuel's Conjecture:

\begin{conjecture}[Schanuel]
  \label{conj:schanuel}For all $z_{1},\dots,z_{n}\in\Cb$,
  \[
    \td(z_{1},\dots,z_{n},\exp(z_{1}),\dots,\exp(z_{n}))\geq\ld_{\Q}(z_{1},\dots,z_{n}).
  \]
\end{conjecture}

We shall denote by ${\bf z}$ a finite tuple of elements of $\Cb$, and by
$\exp({\bf z})$ the tuple of their exponentials, so that Schanuel's Conjecture
can be rewritten as
\[
  \td({\bf z},\exp({\bf z}))\geq\ld_{\Q}({\bf z}).
\]

\begin{prop}
  \label{prop:fg-field-to-linear}Let $p(x,y)\in\Cb[x,y]$ be an irreducible
  polynomial such that
  $\frac{\partial p}{\partial x},\frac{\partial p}{\partial y}\neq0$ and $k$ be
  a finitely generated subfield of $\Cb$.

  If Schanuel's conjecture holds, then all the solutions of \eqref{eq:orig} in
  $\overline{k}$ are contained in some finite-dimensional $\Q$-linear space
  $L\subset\overline{k}$.
\end{prop}
\begin{proof}
  Let $k'$ be the field generated by $k$ and the coefficients of $p$.  If
  Schanuel's conjecture holds, for any ${\bf z}\subset\Cb$ we have
  \[
    \td({\bf z},\exp({\bf z}))\geq\ld_{\Q}({\bf z}).
  \]

  However, if each entry $z$ of ${\bf z}$ is in $\overline{k}$ and
  $p(z,\exp(z))=0$, we also have $\exp(z)\in\overline{k}'$, so that in
  particular
  \[
    \td(\overline{k}')=\td(k')\geq\td({\bf z},\exp({\bf z}))\geq\ld_{\Q}({\bf
      z}).
  \]

  This implies that the solutions of \eqref{eq:orig} in $\overline{k}$ live in a
  $\Q$-linear subspace $L\subset\overline{k}$ of dimension at most
  $\td(\overline{k}')$. Since $k'$ is finitely generated, $L$ is
  finite-dimensional, as desired.
\end{proof}

In particular, if ${\bf b}\subset L$ is a $\Q$-linear basis of $L$, then all the
solutions of \prettyref{eq:orig} in $\overline{k}$ are of the form
$\mathbf{q}\cdot\mathbf{b}$ for some vector ${\bf q}$ with rational
coefficients.

\begin{cor}
  \label{cor:red-to-linear}Let $p(x,y)\in\Cb[x,y]$ be an irreducible polynomial
  such that $\frac{\partial p}{\partial x},\frac{\partial p}{\partial y}\neq0$
  and $k$ be a finitely generated subfield of $\Cb$.

  If Schanuel's Conjecture holds, then there exist $l\in\N$ and
  ${\bf b}\in\Cb^{l}$ with $\Q$-linearly independent entries such that
  \eqref{eq:orig} has only finitely many solutions in $\overline{k}$ if and only
  if
  \begin{equation}
    p({\bf x}\cdot{\bf b},\exp({\bf x}\cdot{\bf b}))=0\label{eq:rational}
  \end{equation}
  has only finitely many solutions in $\Q^{l}$.
\end{cor}

\subsection{\label{sub:red-to-integer}Reduction to integer solutions}

In \cite{Gunaydin2011} it is shown that the rational solutions of polynomial
equations like \eqref{eq:rational} have bounded denominators, so that our
problem becomes one of counting \emph{integer} solutions.  We recall the
original statement in its full form.

Consider some ${\bf b}_{i}\in\Cb^{t}$, where $t\in\N_{>0}$ and $i$ ranges in
$\{1,\dots,s\}$ for some integer $s>1$. We study the rational solutions of the
equation
\begin{equation}
  \sum_{i=1}^{s}q_{i}({\bf x})\exp({\bf x}\cdot{\bf b}_{i})=0,\label{eq:rational-general}
\end{equation}
where the $q_{i}({\bf x})$'s are polynomials in $\Cb[{\bf x}]$. This includes
\eqref{eq:rational} as a special case.

First of all, we exclude the degenerate solutions. A solution ${\bf q}$ of
\eqref{eq:rational-general} is \emph{degenerate} if there is a finite proper
subset $B\subset\{1,\dots,s\}$ such that
\[
  \sum_{i\in B}q_{i}({\bf q})\exp({\bf q}\cdot{\bf b}_{i})=0.
\]
A solution is \emph{non-degenerate} otherwise. Moreover, we project away the
`trivial part' of the solutions given by the ${\bf q}$'s such that
$\exp({\bf q}\cdot{\bf b}_{i})=\exp({\bf q}\cdot{\bf b}_{j})$ for all $i,j$. Let
$V$ be the subspace of $\Q^{t}$ of such $\Q$-linear relations, i.e.,
\[
  V:=\{{\bf q}\in\Q^{t}\,:\,{\bf q}\cdot{\bf b}_{i}={\bf q}\cdot{\bf
    b}_{j}\mbox{ for all }i,j=1,\dots,s\},
\]
and let $\pi':\Q^{t}\to V'$ be the projection onto some complement $V'$ of $V$
in $\Q^{t}$.

With this notation, we have the following.

\begin{thm}[{\cite[Thm.~1.1]{Gunaydin2011}}]
  \label{thm:gunaydin} Given $q_{1},\dots,q_{s}\in\Cb[{\bf x}]$ and
  ${\bf b}_{1},\dots,{\bf b}_{s}\in\Cb^{t}$, there is $N\in\N_{>0}$ such that if
  ${\bf q}\in\Q^{t}$ is a non-degenerate solution of
  \[
    \sum_{i=1}^{s}q_{i}({\bf x})\exp({\bf x}\cdot{\bf b}_{i})=0,
  \]
  then $\pi'({\bf q})\in\frac{1}{N}\Z^{t}$.
\end{thm}

In the case of \eqref{eq:rational}, we can easily deduce that there must be an
$N>0$ such that \emph{all} its rational solutions are in $\frac{1}{N}\Z^{t}$.

\begin{prop}
  \label{prop:red-to-int}Let $p(x,y)\in\Cb[x,y]$ be an irreducible polynomial
  such that $\frac{\partial p}{\partial x},\frac{\partial p}{\partial y}\neq0$
  and ${\bf b}\in\Cb^{l}$ be a vector with $\Q$-linearly independent
  entries. Then there exists an integer $N>0$ such that the rational solutions
  of \eqref{eq:rational} are contained in $\frac{1}{N}\Z^{l}$.

  In particular, there exists a ${\bf b}'\in\Cb^{l}$ with $\Q$-linearly
  independent entries such that \eqref{eq:rational} has only finitely many
  rational solutions if and only if
  \begin{equation}
    p({\bf x}\cdot{\bf b}',\exp({\bf x}\cdot{\bf b}'))=0\label{eq:integer}
  \end{equation}
  has only finitely many integer solutions.
\end{prop}
\begin{proof}
  We can rewrite \eqref{eq:rational} as
  \[
    p({\bf x}\cdot{\bf b},\exp({\bf x}\cdot{\bf b}))=\sum_{i=0}^{d}q_{i}({\bf
      x}\cdot{\bf b})\cdot\exp(i{\bf x}\cdot{\bf b})=0.
  \]

  There are at most finitely many solutions such that
  $q_{i}({\bf x}\cdot{\bf b})=0$ for all $i=0,\dots,d$; indeed, for such
  solutions the value ${\bf x}\cdot{\bf b}$ ranges in a finite set, and since
  the entries of ${\bf b}$ are $\Q$-linearly independent, each value of
  ${\bf x}\cdot{\bf b}$ determines at most one value of ${\bf x}$. Therefore,
  there is a positive integer $N_{0}$ such that these solutions are contained in
  $\frac{1}{N_{0}}\Z^{l}$.

  Consider now the solutions such that $q_{i}({\bf x}\cdot{\bf b})\neq0$ for at
  least one $i$. For each such solution ${\bf q}$, there must be a subset
  $B\subset\{0,\dots,d\}$ with $|B|\geq2$ such that ${\bf q}$ is a
  non-degenerate solution of
  \[
    \sum_{i\in B}q_{i}({\bf x}\cdot{\bf b})\cdot\exp(i{\bf x}\cdot{\bf b})=0.
  \]

  We apply \prettyref{thm:gunaydin} to the equation given by such a $B$. The
  corresponding
  \[
    V_{B}=\{{\bf q}\in\Q^{l}\,:\,{\bf q}\cdot(i{\bf b})={\bf q}\cdot(j{\bf
      b})\mbox{ for all }i,j\in B\}
  \]
  is null, since the entries of ${\bf b}$ are $\Q$-linearly independent and
  $|B|\geq2$; therefore, the non-degenerate solutions lie in
  $\frac{1}{N_{B}}\Z^{l}$ for some positive integer $N_{B}$.

  We now define $N$ as the least common multiple of the various $N_{B}$ and of
  $N_{0}$, so that all the rational solutions are contained in
  $\frac{1}{N}\Z^{l}$. In particular, the rational solutions of
  \eqref{eq:rational} are in bijection with the integer solutions of the
  equation
  \[
    p({\bf x}\cdot{\bf b}/N,\exp({\bf x}\cdot{\bf b}/N))=p({\bf x}\cdot{\bf
      b}',\exp({\bf x}\cdot{\bf b}'))=0,
  \]
  where ${\bf b}':={\bf b}/N$, proving the desired conclusion.
\end{proof}

\subsection{\label{sub:red-to-alg}Reduction to algebraic exponentials}

Using the main result of \cite{Zannier2004}, we can reduce the problem of
finding integer solutions of \eqref{eq:integer} to the case where
$\exp({\bf b}')\subset\oQ^{*}$. In the following, if ${\bf A}$ is a vector
$(a_{1},\dots,a_{n})$ in $\Cb^{n}$ and ${\bf m}$ is a vector
$(m_{1},\dots,m_{n})$ in $\N^{n}$, we write ${\bf A}^{{\bf m}}$ to denote the
product $a_{1}^{m_{1}}\cdot\dots\cdot a_{n}^{m_{n}}$.

As in \cite{Zannier2004}, we start from the equation
\begin{equation}
  \sum_{i=1}^{s}q_{i}({\bf x}){\bf A}_{i}^{{\bf x}}=0\label{eq:zannier}
\end{equation}
where ${\bf A}_{i}\in\Cb^{t}$. In \cite{Laurent1989}, M.\ Laurent showed that in
a precise sense the non-degenerate solutions are not far away from the submodule
\[
  H:=\{{\bf n}\in\Z^{t}\,:\,{\bf A}_{i}^{{\bf n}}={\bf A}_{j}^{{\bf n}}\mbox{
    for all }i,j=1,\dots,s\}.
\]

Note that $H$ is indeed defined similarly to $V$. If the polynomials $q_{i}$ are
constant, then it turns out that the solutions actually lie in a finite union of
translates of $H$. However, this is not a finiteness result, and as mentioned
in~\cite{Gunaydin2011}, even when combined with \prettyref{thm:gunaydin} it is
not sufficient to prove \prettyref{thm:main}.

On the other hand, we may control the solutions using a ``function field''
version of Laurent's theorem as found in \cite{Zannier2004}.  In this version,
consider two finitely generated fields $K\subset F$ such that the polynomials
and the ${\bf A}_{i}$'s are defined over $F$ and $\td_{K}(F)\geq1$. Rather than
looking for zeroes, the theorem makes a statement regarding the set
\[
  S:=\{{\bf n}\in\Z^{t}\,:\,\mbox{the }q_{i}({\bf n}){\bf A}_{i}^{{\bf n}}\mbox{
    are }\overline{K}\mbox{-linearly dependent}\}
\]
which contains at least the solutions of \eqref{eq:zannier}.

If we know that
$\left({\bf A}_{i}{\bf A}_{j}^{-1}\right)^{{\bf n}}\in\overline{K}^{*}$ is true
for all $i,j$, then we can rewrite
$q_{i}({\bf n}){\bf A}_{i}^{{\bf n}}=q_{i}({\bf n}){\bf A}_{1}^{{\bf
    n}}\left({\bf A}_{i}{\bf A}_{1}^{-1}\right)^{{\bf n}}$, and it follows that
the $q_{i}({\bf n}){\bf A}_{i}^{{\bf n}}$'s are $\overline{K}$-linearly
dependent if and only if the $q_{i}({\bf n})$'s are. Moreover, if there is some
$B\subseteq\{1,\dots,s\}$ such that the $q_{i}({\bf n}){\bf A}_{i}^{{\bf n}}$
for $i\in B$ are $\overline{K}$-linearly dependent, we may deduce the analogous
conclusion if
$\left({\bf A}_{i}{\bf A}_{j}^{-1}\right)^{{\bf n}}\in\overline{K}^{*}$ is true
just for $i,j$ varying in $B$. We group the elements of $S$ accordingly.

\begin{defn}
  Let $B$ be a nonempty subset of $\{1,\dots,s\}$. A set $S'\subset\Z^{t}$ is a
  \emph{class relative to $B$} if
  \begin{enumerate}
  \item for each ${\bf n}\in S'$, the elements
    $q_{i}({\bf n}){\bf A}_{i}^{{\bf n}}$ for $i\in B$ are
    $\overline{K}$-linearly dependent;
  \item there is an ${\bf n}_{0}\in S'$ such that for all ${\bf n}\in\Z^{t}$,
    the vector ${\bf n}$ is in $S'$ if and only if it satisfies (1) and for all
    $i,j\in B$ we have
    $({\bf A}_{i}{\bf A}_{j}^{-1})^{{\bf n}-{\bf n}_{0}}\in\overline{K}^{*}$.
  \end{enumerate}
\end{defn}

\begin{thm}[{\cite[Thm.\ 1]{Zannier2004}}]
  \label{thm:zannier}The set $S$ is a union of finitely many classes.
\end{thm}

\begin{rem}
  In the original article, the theorem is only stated when $F$ has transcendence
  degree $1$ over $K$; it is however noted that the arguments would actually
  carry on to larger transcendence degrees (a summary of the few required
  changes is given in \cite[Rmk.\ 3]{Zannier2004}).  The restricted version with
  $\td(F/K)=1$ would also work for our purposes, as it lets us reduce the
  transcendence degree of the exponentials by one; a careful argument with
  specialisations would then let us reduce the transcendence degree of the
  coefficients, so that the theorem may be applied again, ultimately leading to
  algebraic exponentials.

  The original theorem also puts explicit bounds to the number of classes in
  terms of the degrees of the polynomials $q_{i}$ and on their number.
\end{rem}

With this theorem we can reduce our problem regarding \eqref{eq:integer} to the
special case in which the exponentials are contained in $\oQ$.

\begin{prop}
  \label{prop:red-to-alg}Let $p(x,y)\in\Cb[x,y]$ be an irreducible polynomial
  such that $\frac{\partial p}{\partial x},\frac{\partial p}{\partial y}\neq0$
  and ${\bf b}'\in\Cb^{l}$ be a vector with $\Q$-linearly independent
  entries. Then there are finitely many irreducible polynomials
  $r_{m}(x,y)\in\Cb[x,y]$ and a vector ${\bf c}$ with $\Q$-linearly independent
  entries such that
  $\frac{\partial r_{m}}{\partial x},\frac{\partial r_{m}}{\partial y}\neq0$,
  $\exp({\bf c})\subset\oQ^{*}$, and such that \eqref{eq:integer} has only
  finitely many integer solutions if and only if each equation
  \begin{equation}
    r_{m}({\bf x}\cdot{\bf c},\exp({\bf x}\cdot{\bf c}))=0\label{eq:integer-alg}
  \end{equation}
  has only finitely many integer solutions.
\end{prop}
\begin{proof}
  Let us write
  \[
    p({\bf x}\cdot{\bf b}',\exp({\bf x}\cdot{\bf b}'))=\sum_{i=0}^{d}p_{i}({\bf
      x}\cdot{\bf b}')\exp(i{\bf x}\cdot{\bf b}')=\sum_{i=0}^{d}q_{i}({\bf
      x})\exp(i{\bf x}\cdot{\bf b}')=0.
  \]

  This is a special case of \eqref{eq:zannier}, so we can apply
  \prettyref{thm:zannier}.  We choose as $F$ a field of definition of $p_{i}$,
  ${\bf b}'$ and $\exp({\bf b}')$, and we pick $K=\Q$. We may assume that
  $\td_{\Q}(F)\geq1$, otherwise the conclusion is trivial as we would have
  $\exp({\bf b}')\subset\oQ$.  Note that $\td_{\Q}(F)\geq1$ is actually always
  true, since at least one between ${\bf b}'$ and $\exp({\bf b}')$ must contain
  a transcendental element by the Hermite-Lindemann-Weierstrass theorem
  \cite{Hermite1873,VonLindemann1882,Weierstrass1885}.

  Let us drop the terms such that $q_{i}\equiv0$, as they do not contribute to
  the sum. As before, there are at most finitely many solutions ${\bf n}$ such
  that $q_{i}({\bf n})=p_{i}({\bf n}\cdot{\bf b}')=0$ for some non-zero $q_{i}$,
  because the entries of ${\bf b}'$ are $\Q$-linearly independent. The remaining
  solutions are such that the corresponding terms in the sum are all non-zero,
  but $\oQ$-linearly dependent, since their sum is $0$. By
  \prettyref{thm:zannier}, these solutions are contained in finitely many
  classes, which means that there exist ${\bf n}_{1}$, $\dots$, ${\bf n}_{k}$
  and $B_{1}$, $\dots$, $B_{k}$ such that for every such solution ${\bf n}$
  there is some $m$ satisfying
  \[
    \exp((i-j)({\bf n}-{\bf n}_{m})\cdot{\bf b}')\in\oQ^{*}
  \]
  for all $i$, $j$ in $B_{m}$. Since no term vanishes, we must have
  $|B_{m}|\geq2$, and the latter condition becomes equivalent to
  \[
    \exp(({\bf n}-{\bf n}_{m})\cdot{\bf b}')\in\oQ^{*}.
  \]

  Let ${\bf c}$ be a $\Z$-linear basis of
  $\log(\oQ^{*})\cap\mathrm{span}_{\Z}({\bf b}')$, so that for all of the above
  vectors ${\bf n}-{\bf n}_{m}$ we have
  $({\bf n}-{\bf n}_{m})\cdot{\bf b}'={\bf n}'\cdot{\bf
    c}\in\mathrm{span}_{\Z}({\bf c})$ for some ${\bf n}'\in\Z^{l'}$. Each
  ${\bf n}'$ is an integer solution of
  \[
    r_{m}({\bf x}\cdot{\bf c},\exp({\bf x}\cdot{\bf c})):=p({\bf n}_{m}\cdot{\bf
      b}'+{\bf x}\cdot{\bf c},\exp({\bf n}_{m}\cdot{\bf b}'+{\bf x}\cdot{\bf
      c}))=0
  \]
  for some $m$; moreover, the map
  ${\bf n}'\mapsto({\bf n}-{\bf n}_{m})\mapsto{\bf n}$ is clearly injective and,
  as $m$ varies, it covers all the integer solutions of \eqref{eq:integer},
  except at most finitely many ones.  Therefore, \eqref{eq:integer} has only
  finitely many integer solutions if and only if each of the above equations
  have only finitely many integer solutions, as desired.
\end{proof}

\subsection{\label{sub:finiteness}Finiteness}

We can finally prove that the integer solutions of \eqref{eq:integer-alg} are
only finitely many. In order to prove that, we use some classical results about
the arithmetic of $\oQ$, and more specifically the properties of the logarithmic
Weil height $\wht:\oQ^{*}\to\R_{\geq0}$.

We just recall the following facts about the function $h$. Let
$\boldsymbol{\gamma}\in(\oQ^{*})^{t}$ and ${\bf m}\in\Z^{t}$. We denote by
$|{\bf m}|_{1}$ the $1$-norm of ${\bf m}$. There are positive numbers
$a_{1}=a_{1}(\boldsymbol{\gamma})$, $a_{2}=a_{2}(\boldsymbol{\gamma})$,
$a_{3}=a_{3}(\boldsymbol{\gamma})$ depending on $\boldsymbol{\gamma}$ only such
that:

\begin{enumerate}
\item
  $h({\bf m}\cdot\boldsymbol{\gamma})\leq a_{1}(\boldsymbol{\gamma})\log|{\bf
    m}|_{1}$;
\item
  $\wht(\boldsymbol{\gamma}^{{\bf m}})\leq a_{2}(\boldsymbol{\gamma})|{\bf
    m}|_{1}$;
\item if the entries of $\boldsymbol{\gamma}$ are multiplicatively independent,
  then
  $\wht(\boldsymbol{\gamma}^{{\bf m}})\geq a_{3}(\boldsymbol{\gamma})|{\bf
    m}|_{1}$.
\end{enumerate}

Moreover, if $f\in\oQ[x,y]$ is such that $\frac{\partial f}{\partial x}\neq0$,
there is $a_{4}=a_{4}(f)$ depending on $f$ only such that if
$\alpha,\beta\in\oQ^{*}$ and $f(\alpha,\beta)=0$, then
$\wht(\alpha)\leq c_{4}(f)h(\beta)$.  We also recall that if $n\in\Z^{*}$, then
$h(n)=\log|n|$.

\begin{prop}
  \label{prop:finite-alg}Let $r_{m}(x,y)\in\Cb[x,y]$ be an irreducible
  polynomial such that
  $\frac{\partial r_{m}}{\partial x},\frac{\partial r_{m}}{\partial y}\neq0$ and
  ${\bf c}\in\Cb^{l}$ be a vector with $\Q$-linearly independent entries such
  that $\exp({\bf c})\subset\oQ^{*}$.

  If Schanuel's Conjecture holds, then \eqref{eq:integer-alg} has only finitely
  many integer solutions.
\end{prop}
\begin{proof}
  We distinguish two cases. Let $K$ be a finitely generated field containing
  ${\bf c}$, $\exp({\bf c})$ and the coefficients of $r_{m}$.

  If $\mathrm{span}_{\Q}({\bf c})$ does not contain $2\pi i$, then
  $\exp({\bf c})$ is a multiplicatively independent set. We pick a
  specialisation $\sigma:K\to\oQ\cup\{\infty\}$ such that the coefficients of
  $r_{m}$ and of ${\bf c}$ become non-zero elements of $\oQ$.  Let
  $r_{m}^{\sigma}\in\oQ[x,y]$ be the specialisation of the polynomial
  $r_{m}$. Note that
  $\frac{\partial r_{m}^{\sigma}}{\partial x},\frac{\partial
    r_{m}^{\sigma}}{\partial y}\neq0$.  For any ${\bf n}$, if $\alpha$ is a root
  of $r_{m}^{\sigma}(x,\exp({\bf n}\cdot{\bf c}))$, then
  $\wht(\alpha)\leq a_{4}(r_{m}^{\sigma})\wht(\exp({\bf n}\cdot{\bf c}))$.
  Symmetrically, $\exp({\bf n}\cdot{\bf c})$ must be a root of
  $r_{m}^{\sigma}(\alpha,y)\in\oQ[y]\setminus\{0\}$, hence
  $\wht(\exp({\bf n}\cdot{\bf c}))\leq a_{4}(r_{m}^{\sigma})\wht(\alpha)$ as
  well.

  Now note that the specialisation $\sigma$ must leave
  $\exp({\bf n}\cdot{\bf c})\in\oQ^{*}$ fixed, so that if ${\bf n}$ is a
  solution of \eqref{eq:integer-alg}, then
  \[
    r_{m}^{\sigma}(\sigma({\bf n}\cdot{\bf c}),\exp({\bf n}\cdot{\bf c}))=0,
  \]
  hence
  \[
    \wht(\exp({\bf n}\cdot{\bf c}))\leq a_{4}(r_{m}^{\sigma})\wht(\sigma({\bf
      n}\cdot{\bf c}))\leq a_{4}(r_{m}^{\sigma})a_{1}({\bf c})\log|{\bf n}|_{1}.
  \]
  Since $\exp({\bf c})$ is multiplicatively independent, this implies that
  \[
    a_{3}(\exp({\bf c}))|{\bf n}|_{1}\leq\wht(\exp({\bf n}\cdot{\bf c}))\leq
    a_{4}(r_{m}^{\sigma})a_{1}({\bf c})\log|{\bf n}|_{1},
  \]
  and therefore that there are only finitely many such ${\bf n}$, as desired.

  If $\mathrm{span}_{\Q}({\bf c})$ contains $2\pi i$, we may assume, without
  loss of generality, that the first coordinate of ${\bf c}$ is $2\pi i/N$ for
  some integer $N$; after splitting \eqref{eq:integer-alg} into $N$ different
  equations, we may directly assume that the first coordinate of ${\bf c}$ is
  $2\pi i$ itself. Write ${\bf c}=(2\pi i)^{\frown}{\bf c}'$, where $\frown$
  indicates vector concatenation, so that ${\bf c'}$ is the vector containing
  all the entries of ${\bf c}$ except for the first one.

  Assuming Schanuel's Conjecture, we deduce that the entries of ${\bf c}$ are
  algebraically independent. Let $F$ be the field generated by ${\bf c}$,
  $\exp({\bf c})$ and the coefficients of $r_{m}$. Let $K'$ be the field
  generated by ${\bf c}'$ and $\exp({\bf c}')$ only. Since the entries of
  ${\bf c}$ are algebraically independent, $2\pi i\notin K'$. We can then easily
  find some $K\supseteq K'$ such that $F/K$ is a finitely generated geometric
  extension (i.e., $K$ is relatively algebraically closed in $F$) of
  transcendence degree one, while $2\pi i\notin K$.

  Let $\sigma:F\to K\cup\{\infty\}$ be a specialisation of $F$ which leaves $K$
  fixed and such that $\sigma(2\pi i)=0$. If we multiply $r_{m}$ by a suitable
  element of $F$, we may assume that the specialisation $r_{m}^{\sigma}$ is a
  non-zero polynomial in $K[x,y]$. Let ${\bf n}$ is an integer solution of
  \eqref{eq:integer-alg} written as ${\bf n}=(n_{0})^{\frown}{\bf n}'$, so that
  ${\bf n}\cdot{\bf c}=2\pi in_{0}+{\bf n}'\cdot{\bf c}'$.  Then
  \[
    \sigma(r_{m}({\bf n}\cdot{\bf c},\exp({\bf n}\cdot{\bf
      c})))=r_{m}^{\sigma}({\bf n}'\cdot{\bf c}',\exp({\bf n}'\cdot{\bf c}'))=0.
  \]
  Again, the specialisation leaves $\exp({\bf n}'\cdot{\bf c})\in\oQ^{*}$ fixed,
  while any mention of $n_{0}$ disappears since $\sigma(n_{0}\cdot2\pi i)=0$ and
  $\exp(n_{0}\cdot2\pi i)=1$. In this way, we have reduced to the case where
  $\mathrm{span}_{\Q}({\bf c}')$ does not contain $2\pi i$.

  If the polynomial $r_{m}^{\sigma}$ satisfies
  $\frac{\partial r_{m}^{\sigma}}{\partial x},\frac{\partial
    r_{m}^{\sigma}}{\partial y}\neq0$, we may reapply the previous argument with
  the heights and deduce that there are at most finitely many vectors
  ${\bf n}'$. Otherwise, if $r_{m}^{\sigma}$ actually lives in $\Cb[x]$ or
  $\Cb[y]$, then it has finitely many solutions because ${\bf c}'$ and
  $\exp({\bf c}')$ are respectively $\Q$-linearly independent and
  multiplicatively independent and $r_{m}^{\sigma}\neq0$. In any case, the
  vectors ${\bf n}'$ are at most finitely many.

  It remains to check $n_{0}$. It must be an integer solution of
  \[
    r_{m}(2\pi ix+{\bf n}'\cdot{\bf c}',\exp(2\pi ix+{\bf n}'\cdot{\bf
      c}'))=r_{m}(2\pi ix+{\bf n}'\cdot{\bf c}',\exp({\bf n}'\cdot{\bf c}'))=0.
  \]
  Since $r_{m}$ is irreducible in $\Cb[x,y]$ and
  $\frac{\partial r_{m}}{\partial x}\neq0$, the above equation is never trivial
  for any choice of ${\bf n}'$, and therefore it has only finitely many
  solutions for each possible ${\bf n}'$. Therefore, there are at most finitely
  many integer solutions ${\bf n}$, as desired.
\end{proof}

Chaining together all of the above statements, we obtain that if Schanuel's
conjecture holds, then \eqref{eq:rational} has only finitely many rational
solutions. This is a conditional version of \prettyref{thm:finite}; an
unconditional proof will be given in the appendix.

\begin{prop}
  \label{prop:finite-rational}Let $p(x,y)\in\Cb[x,y]$ be an irreducible
  polynomial such that
  $\frac{\partial p}{\partial x},\frac{\partial p}{\partial y}\neq0$ and
  ${\bf b}\in\Cb^{l}$ be a vector with $\Q$-linearly independent entries.

  If Schanuel's Conjecture holds, then \eqref{eq:rational} has only finitely
  many rational solutions.
\end{prop}
\begin{proof}
  It suffices to apply \prettyref{prop:red-to-int}, \prettyref{prop:red-to-alg},
  and \prettyref{prop:finite-alg} in sequence.
\end{proof}

This is enough to prove the main theorem.

\begin{proof}[Proof of \prettyref{thm:main}]
  By \prettyref{cor:red-to-linear}, Schanuel's Conjecture implies that, given a
  finitely generated field $k$, the solutions of \eqref{eq:orig} in
  $\overline{k}$ are in bijection with the rational solutions of
  \eqref{eq:rational}, which are at most finitely many by
  \prettyref{prop:finite-rational}, again assuming Schanuel's
  conjecture. However, \eqref{eq:orig} has infinitely many solutions in $\Cb$,
  and therefore at least one is not in $\overline{k}$, as desired.
\end{proof}

\section{Appendix}

\begin{center}
By V.\ Mantova and U.\ Zannier
\end{center}

\bigskip

The purpose of this appendix is to give a proof of the conclusion of
\prettyref{prop:finite-alg} without assuming Schanuel's Conjecture.  Following
all the previous implications, this gives an unconditional proof of
\prettyref{thm:finite}. It is not surprising that we still use a deep theorem of
transcendental number theory due to Alan Baker, which is a special true case of
the conjecture.

We shall prove the following.

\begin{prop}
  \label{prop:finite-alg-uncond}Let $r_{m}(x,y)\in\Cb[x,y]$ be an irreducible
  polynomial such that
  $\frac{\partial r_{m}}{\partial x},\frac{\partial r_{m}}{\partial y}\neq0$ and
  ${\bf c}\in\Cb^{l}$ be a vector with $\Q$-linearly independent entries such
  that $\exp({\bf c})\subset\oQ^{*}$. Then \eqref{eq:integer-alg} has only
  finitely many integer solutions.
\end{prop}

Informally, the idea is to consider \eqref{eq:integer-alg} as an equation
between algebraic functions
\[
  {\bf x}\cdot{\bf c}=\psi(\exp({\bf x}\cdot{\bf c})),
\]
where $\psi$ is an element of $\overline{\Cb(y)}$ such that $r_{m}(\psi,y)=0$.
We then exploit the algebraic properties of ${\bf c}$ to show that the entries
of ${\bf x}$ are actually themselves algebraic functions of
$\exp({\bf x}\cdot{\bf c})$. Once we know this, we can look at the heights of
the functions to deduce the finiteness similarly to the proof of
\prettyref{prop:finite-alg}.

In place of Schanuel's Conjecture, we use the following deep theorem of
transcendental number theory by Alan Baker.

\begin{thm}[A.\ Baker \cite{Baker1966,Baker1967,Baker1967a}]
  \label{thm:Baker}If $\alpha_{1},\dots,\alpha_{n}\in\log(\oQ^{*})$ are
  $\Q$-linearly independent, then they are also $\oQ$-linearly independent.
\end{thm}

It is easy to see that this is one of the consequences of Schanuel's
Conjecture. In our case, it shows that our vector ${\bf c}$ is actually
$\oQ$-linearly independent.

Before going on with the proof of \prettyref{prop:finite-alg-uncond}, we recall
a couple of classical results about function fields.

\begin{notation}
  Let ${\bf x},y$ be independent variables over $\oQ$. Let $K$ be a finite
  extension of $\oQ({\bf x})$ and $E$ be a finite extension of $\oQ(y)$. We look
  at $E$ as the function field of a normal, projective and irreducible curve
  $\mathcal{D}$ over $\oQ$. We call $KE$ their compositum inside an algebraic
  closure of $\Q({\bf x},y)$.
\end{notation}

Recall that $KE$ can be seen as the quotient field of the ring of the finite
sums $a_{1}b_{1}+\dots+a_{m}b_{m}$, with $a_{i}\in K$ and $b_{i}\in E$; the ring
itself, both as an $E$-vector space and as a $K$-vector space, is isomorphic to
$E\otimes_{\oQ}K$ because $K$ and $E$ are linearly disjoint over
$\oQ$. Moreover, each specialisation from $E/\oQ$ to $\oQ$ extends (uniquely) to
a specialisation from $KE$ to $K$ that leaves the elements of $K$ fixed.

Let $f(z)\in KE[z]$, where $z$ is a further independent variable.  For all
$P\in\mathcal{D}(\oQ)$ except for at most finitely many points, if $\Oc_{P}$ is
the local ring of $\mathcal{D}(\oQ)$ at $P$, and $\Oc_{P}^{K}$ the local ring of
$\mathcal{D}(K)$ at $P$, we have that $f\in\Oc_{P}^{K}[z]$; it suffices to avoid
the (finitely many) poles of the coefficients of $f$. When $f\in\Oc_{P}^{K}[z]$,
we shall call $f_{P}\in K[z]$ the specialisation of $f$ at $P$.

\begin{prop}
  \label{prop:Bertini}Let $f\in KE[z]$ be a monic polynomial irreducible in
  $K\overline{E}[z]$. Then the polynomial $f_{P}$ is well-defined and
  irreducible in $K[z]$ for all $P\in\mathcal{D}(\oQ)$ except at most finitely
  many points.
\end{prop}

\begin{rem}
  In the case when $K=\oQ(x)$, the result is a special case of a well known
  theorem usually denominated ``Bertini-Noether theorem'' \cite[Prop.\
  9.4.3]{Fried2008}.  We have not been able to locate in the literature this
  slightly more general version, so we provide a proof by reduction to the
  Bertini-Noether theorem.
\end{rem}
\begin{proof}
  The polynomial $f_{P}$ is well-defined as long as we choose
  $P\in\mathcal{D}(\oQ)$ outside the poles of the coefficients of $f$, which are
  finitely many. Therefore, we may assume that $f_{P}$ is well-defined.

  Let $F$ be the field extension of $KE$ generated by a root $\alpha$ of
  $f$. Choose a primitive element $\beta$ of $F/E({\bf x})$ and let
  $g({\bf x},z)\in E({\bf x})[z]$ be its minimal polynomial over $E({\bf
    x})$. Without loss of generality, we may assume that $g$ is actually a monic
  irreducible polynomial in $E[{\bf x},z]$.

  Since $K$ is linearly disjoint from $E$ over $\oQ$, we have
  $[K:\oQ({\bf x})]=[KE:E({\bf x})]=[K\overline{E}:\overline{E}({\bf x})]$;
  moreover, since $f$ is irreducible over $K\overline{E}$, we have
  $[F:KE]=[F\overline{E}:K\overline{E}]$. In particular, we have
  \[
    [F:E({\bf x})]=[F:KE]\cdot[KE:E({\bf
      x})]=[F\overline{E}:K\overline{E}]\cdot[K\overline{E}:\overline{E}({\bf
      x})]=[F\overline{E}:\overline{E}({\bf x})].
  \]

  This implies that the polynomial $g$ is absolutely irreducible as a polynomial
  in several variables over $E$. By the Bertini-Noether theorem, in the form
  \cite[Prop.\ 9.4.3]{Fried2008}, for all the points $P\in\mathcal{D}(\oQ)$
  except at most finitely many ones, $g_{P}$ is well-defined (i.e.,
  $g\in\Oc_{P}[{\bf x},z]$) and absolutely irreducible.

  Now take a $P$ such that $g_{P}$ is absolutely irreducible, let $\Oc_{P}^{F}$
  be an extension of $\Oc_{P}^{K}$ to a valuation subring of $F$ and let
  $\sigma:\Oc_{P}^{F}\to\overline{K}$ be the corresponding specialisation which
  extends the specialisation at $P$. We note that both $\alpha$ and $\beta$ are
  in $\Oc_{P}^{F}$, and except for at most finitely many choices of $P$,
  $\oQ({\bf x},\sigma(\beta))=K(\sigma(\alpha))$.  Since
  $g_{P}(\sigma(\beta))=0$, we have that
  $[K(\sigma(\alpha)):\oQ({\bf x})]=[F:E({\bf x})]$.

  Therefore, the degrees of the subextensions are preserved as well.  In
  particular, $[K(\sigma(\alpha)):K]=[F:KE]$. Since $f_{P}(\sigma(\alpha))=0$,
  this shows that $f_{P}$ must be irreducible for all but finitely many
  $P\in\mathcal{D}(\oQ)$, as desired.
\end{proof}

Now, let $\psi_{1},\dots,\psi_{m}$ be functions in $KE$. As before, we may
define their specialisations $\psi_{j}^{P}$ for all $P\in\mathcal{D}(\oQ)$
except at most finitely many points.

\begin{prop}
  \label{prop:lin-indep}If $\psi_{1},\dots,\psi_{m}\in KE$ are $E$-linearly
  independent, then for all $P\in\mathcal{D}(\oQ)$ except at most finitely many
  points, the specialisations $\psi_{1}^{P},\dots,\psi_{n}^{P}$ are well-defined
  and $\oQ$-linearly independent.
\end{prop}
\begin{proof}
  Let $\psi_{1}$, $\dots$, $\psi_{m}$ be functions as in the hypothesis.  Since
  $KE$ is the quotient field of the finite sums $a_{1}b_{1}+\dots+a_{k}b_{k}$
  with $a_{i}\in K$ and $b_{i}\in E$, after multiplying by a common denominator
  we may assume that each function $\psi_{j}$ is of the form
  \[
    \psi_{j}=a_{1}b_{1}+\dots+a_{k}b_{k}
  \]
  with $a_{i}\in K$ and $b_{i}\in E$. In particular, we may find a
  $\oQ$-linearly independent set $d_{1},\dots,d_{l}\in K$ and coefficients
  $b_{jk}\in E$ such that
  \[
    \psi_{j}=\sum_{k=1}^{l}d_{k}b_{jk}.
  \]

  Since $d_{1}$, $\dots$, $d_{l}$ are $\oQ$-linearly independent, they are also
  $E$-linearly independent, as $E$ and $K$ are linearly disjoint. The fact that
  the functions $\psi_{j}$ are $E$-linearly independent translates to the fact
  that the matrix $(b_{jk})_{j,k}$ must have rank $m$. But then the matrix
  $(b_{jk}^{P})_{j,k}$ is well-defined and has rank $m$ for all but finitely
  many points, in which case the functions $\psi_{j}^{P}$ are $\oQ$-linearly
  independent, as desired.
\end{proof}

Now we are able to prove \prettyref{prop:finite-alg-uncond}.

\begin{proof}[Proof of \prettyref{prop:finite-alg-uncond}]
  If $2\pi i\notin\mathrm{span}_{\Q}({\bf c})$, the argument in the proof of
  \prettyref{prop:finite-alg} already shows that the desired result holds
  unconditionally, so we may assume that $2\pi i\in\mathrm{span}_{\Q}({\bf c})$.
  As in the previous proof, up to replacing $r_{m}$ with finitely many
  polynomials of the same shape, we may assume that the first coordinate of
  ${\bf c}$ is $2\pi i$. We write again ${\bf c}=(2\pi i)^{\frown}{\bf c}'$.

  Suppose first that the solutions ${\bf n}=(n_{0})^{\frown}{\bf n}'$ are such
  that the vectors ${\bf n}'$ are only finitely many. As in the proof of
  \prettyref{prop:finite-alg}, for each such given ${\bf n}'$,
  \eqref{eq:integer-alg} becomes a polynomial in $n_{0}$, so that the solutions
  are only finitely many, as desired.

  Now assume by contradiction that the vectors ${\bf n}'$ are infinitely
  many. We also have that
  $\exp({\bf n}\cdot{\bf c})=\exp({\bf n}'\cdot{\bf c}')$ takes infinitely many
  different values on the solutions, since the entries of $\exp({\bf c}')$ are
  multiplicatively independent.

  Let $K$ be the field generated by $\oQ$, ${\bf c}$ and the coefficients of
  $r_{m}$. By Baker's Theorem \ref{thm:Baker}, the entries of ${\bf c}$ are
  $\oQ$-linearly independent. Let $z$ and $y$ be new elements algebraically
  independent over $K$, and look at $r_{m}(z,y)$ as a polynomial in $z$ with
  coefficients in $K(y)$. After taking a finite extension $E/\oQ(y)$, we may
  split $r_{m}(z,y)$ into finitely many factors that are irreducible in
  $K\overline{E}[z]$. Let $\mathcal{D}$ a normal, irreducible and projective
  curve over $\oQ$ whose function field is $E$.

  If ${\bf n}$ is a solution of \eqref{eq:integer-alg}, there is a point
  $P\in\mathcal{D}(\oQ)$ such that $r_{m}(z,y_{P})$ has a linear factor of the
  form $(z-{\bf n}\cdot{\bf c})$ and $y_{P}=\exp({\bf n}\cdot{\bf c})$.  Since
  we are assuming that $\exp({\bf n}\cdot{\bf c})=\exp({\bf n}'\cdot{\bf c})$
  takes infinitely many values as ${\bf n}$ varies on the solutions, there are
  infinitely many points such that $r_{m}$ has a linear factor.  By
  \prettyref{prop:Bertini}, except for at most finitely many of these points,
  each such linear factor must be the specialisation of a linear factor of
  $r_{m}(z,y)$ over $KE$.

  In particular, there must be a function $\phi\in KE$ such that $(z-\phi)$ is a
  linear factor of $r_{m}(z,y)$, and such that for infinitely many points
  $P\in\mathcal{D}(\oQ)$ there is a solution ${\bf n}_{P}$ of
  \eqref{eq:integer-alg} such that
  \[
    y_{P}=\exp({\bf n}_{P}\cdot{\bf c}),\quad\phi_{P}={\bf n}_{P}\cdot{\bf c}.
  \]
  As before, write ${\bf n}_{P}=(n_{0,P})^{\frown}{\bf n}_{P}'$. Let
  $\mathcal{P}$ be the (infinite) set of such points $P$; without loss of
  generality, we may further assume that if $P\neq Q\in\mathcal{P}$, then
  ${\bf n}_{P}'\neq{\bf n}_{Q}'$. Since $K$ and $E$ are linearly disjoint over
  $\oQ$, ${\bf c}$ is $E$-linearly independent. Since $\phi_{P}$ is
  $\oQ$-linearly dependent on ${\bf c}$ for all $P\in\mathcal{P}$,
  \prettyref{prop:lin-indep} implies that $\phi$ is $E$-linearly dependent on
  ${\bf c}$, so that
  \[
    \phi=\boldsymbol{\psi}\cdot{\bf c}
  \]
  for some vector $\boldsymbol{\psi}$ of algebraic functions in $E$.  Moreover,
  when $\phi_{P}={\bf n}_{P}\cdot{\bf c}$ we must have
  $\boldsymbol{\psi}_{P}={\bf n}_{P}$.  Let us write
  $\boldsymbol{\psi}=(\psi_{0})^{\frown}\boldsymbol{\psi}'$.

  We now use the logarithmic Weil height again, as in
  \prettyref{sub:finiteness}.  Let $\psi$ be an entry of
  $\boldsymbol{\psi}'$. If $\psi$ is non-constant, we have that $y$ is algebraic
  over $\oQ(\psi)$, hence $f(y,\psi)=0$ for some $f\in\oQ[z,w]$ such that
  $\frac{\partial f}{\partial z}\neq0$.  In particular, $f(y_{P},\psi_{P})=0$,
  hence $\wht(y_{P})\leq a_{4}(f)\wht(\psi_{P})$.  Therefore, for each
  $P\in\mathcal{P}$ we have
  \[
    \wht(y_{P})=\wht(\exp({\bf n}_{P}\cdot{\bf c}))=\wht(\exp({\bf
      n}_{P}'\cdot{\bf c}'))\geq a_{3}(\exp({\bf c}'))|{\bf n}_{P}'|_{1},
  \]
  while we also have
  \[
    h(y_{P})\leq a_{4}(f)\wht(\psi_{P})=a_{4}(f)\log|\psi_{P}|\leq a_{4}(f)\log|{\bf n}_{P}'|_{1}.
  \]

  This implies that the range of $\psi$ as $P$ varies in $\mathcal{P}$ is
  finite. Since $\mathcal{P}$ is infinite, this implies that $\psi$ is actually
  constant. Since this must be true for any entry of $\boldsymbol{\psi}'$, the
  vector $\boldsymbol{\psi}'$ itself is constant, which implies that the vectors
  ${\bf n}_{P}'$ are only finitely many, and in particular that $\mathcal{P}$ is
  finite, a contradiction.
\end{proof}

The above proposition now yields the unconditional proof of
\prettyref{thm:finite}.

\begin{proof}[Proof of \prettyref{thm:finite}]
  It suffices to apply \prettyref{prop:red-to-int}, \prettyref{prop:red-to-alg},
  and \prettyref{prop:finite-alg-uncond} in sequence.
\end{proof}

\end{document}